\documentclass[a4paper,12pt]{article}

\usepackage{amssymb}
\usepackage{amsmath}

\usepackage{amsfonts}

\usepackage[plainpages=false]{hyperref}
\usepackage{amsfonts,latexsym,rawfonts,amsmath,amssymb,amsthm,mathrsfs}
\usepackage{amsmath,amssymb,amsfonts,latexsym,lscape,rawfonts}

\usepackage[all]{xy}
\usepackage{eufrak}
\usepackage{makeidx}         
\usepackage{graphicx,psfrag}

\usepackage{array,tabularx}

\usepackage{setspace}
\usepackage{geometry}

\newtheorem{thm}{Theorem}

\newtheorem{rmk}[thm]{Remark}

\theoremstyle{definition}

\def \R {\mathbb R}

\title{Local solution to the $G_{2}-$monopole equation with prescribed tangent cone and $G_{2}-$structure}

{\small{\author{Yuanqi Wang\thanks{Department of Mathematics, Stony Brook University, NY, USA.\ ywang@scgp.stonybrook.edu.}}}

\date{\vspace{-5ex}}
\begin{document}
\maketitle Given a  $G_{2}-$structure  $(\phi,\psi)$, on the $G_{2}-$monopole equation
\begin{equation}\label{equ monopole equation}F_{A^{\star}}\wedge \psi+\star_{\phi}(d_{A^{\star}}u)=0\  \textrm{of  a connection}\ A^{\star}\ \textrm{and a bundle-valued}\ 0- \textrm{form}\ u,   \end{equation}
the following theorem is true.
\begin{thm}\label{Thm local solution}Let $B_{O}(R)\subset \R^{7}$ be an arbitrary ball centred at the origin. For any smooth  $G_{2}-$structure $(\phi,\psi)$ defined over $B_{O}(R)$, and any smooth $SO(m)-$bundle $\eta \rightarrow \mathbb{S}^{6}$ equipped with a Hermitian-Yang-Mills connection $A_{0}$, there exists a $G_{2}-$monopole which is defined in a smaller ball and  asymptotic to $A_{0}$ exponentially at $O$.
\end{thm}
\begin{rmk} Not every singular elliptic equation admits a  local solution. For example,  Brezis-Cabr\'e \cite{Brezis} showed that  the equation $\Delta u=-\frac{u^{2}+1}{|x|^2}$ does not admit any solution defined near the origin. In contrast, our theorem says that the  singular $G_{2}-$monopole equation is always locally solvable. In particular, for any smooth $G_{2}-$structure defined near $O$, it yields a $G_{2}-$monopole tangent to the canonical connection on $\mathbb{S}^{6}$ (see \cite{DerekHarland} and \cite{XuFeng}). We hope this could help to construct $G_{2}-$instantons with point singularities on a closed $7-$manifold.
\end{rmk}
\begin{rmk} We expect the local solution to be highly non-unique. There exists a solution whose exponential rate is arbitrarily close to $1$ [see (\ref{equ exponential convergence after  pulling back wrt nablaA0}) and the discussion below (\ref{equ 2})].
\end{rmk}
The monopole equation in $G_{2}-$setting first appeared in \cite{DonSegal}  by Donaldson-Segal. For highly-related  later development  on  $G_{2}$ or other kinds of monopoles (instantons), we refer the interested readers to the work done by Sa Earp-Walpuski \cite{SaEarpWalpuski}, Walpuski \cite{Walpuski1}, Oliveira (\cite{Oliveira}, \cite{Oliveira1}, \cite{Oliveira2}, \cite{Oliveira3}), Foscolo \cite{Foscolo},  Charbonneau-Harland \cite{DerekHarland}, Xu \cite{XuFeng}, and the references therein.
\begin{proof}[Proof of Theorem \ref{Thm local solution}:] Near $O$,  $(\phi,\psi)$ is a small perturbation of $[\phi(O),\psi(O)]$. Using a sophisticated version of the rescaling in page 6-9 of \cite{MySchauder}, we show in the following that Theorem \ref{Thm local solution} is a direct corollary of  Theorem 1.13 in \cite{Myself2016}.

 Let the coordinate vector of $B_{O}(R)\subset \R^{7}$ be $v=\left[\begin{array}{c}v_{1} \\ \vdots \\  v_{7} \\ \end{array}\right]$. All the balls below are centred at $O$.  By Lemma 3.7 in \cite{SW}, there exists a linear transformation $L$ such that under the new coordinate $y=Lv$,   $\phi(O)$ is the Euclidean $G_{2}-$form i.e.
\begin{equation}\label{equ phi 0}\phi(O)=dy^{123}-dy^{145}-dy^{167}-dy^{246}+dy^{257}-dy^{347}-dy^{356}.\end{equation} 
It suffices to work under the $y-$coordinate,  under which  $\phi$ is defined in $B(R_{0})$ for some $R_{0}>0$ depending on $R$ and $L$.  We bring in the bundle $\eta$  as defined over $S^{6}(1)$ (the unit sphere), and then view  it as a bundle over $\R^{7}\setminus O$  pulled back from the natural spherical projection (Remark 2.3 in \cite{Myself2016}). The connection $A_{0}$ is pulled-back to be a $G_{2}-$instanton on $\R^{7}\setminus O$ with respect to $\phi(O)$. 
 
  We write $\phi=\Sigma_{ijk}\phi_{ijk}dy^{ijk}$. 
Let $\Gamma$ denote the map $x=\Gamma (y)=\lambda y$ from $B(\frac{1}{4\lambda})$ to $B(\frac{1}{4})$. To the $x-$coordinate, exactly as in the previous paragraph, we can also  pull back   the bundles $\eta$, $ad\eta$, and $A_{0}$ (denoted the same as in $y-$coordinate). Since they are objects on $\mathbb{S}^{6}$,   they are  invariant under $\Gamma$. Let 
\begin{equation}\label{equ 1}\widetilde{\phi}=\Sigma_{ijk}\Gamma^{-1,\star}(\phi_{ijk})dx^{ijk}= \lambda^{3}\Gamma^{-1,\star}\phi,\ \textrm{where}\ \phi_{ijk}\ \textrm{is the same as above}.
\end{equation}
Let $c_{\phi}$ denote $C\Sigma_{ijk}|\phi_{ijk}-\phi_{ijk}(O)|_{C^{5}_{y}[B(\frac{1}{4\lambda})]}$, where $C$ is a proper universal constant (which could be different in various context), and $C^{5}_{y}$  means  the $C^{5}-$norm in $y-$coordinate. By chain-rule we have for any $x$ that 
$$|\nabla_{x}^{k}(\Gamma^{-1,\star}\phi_{ijk}-\phi_{ijk}(O))|(x)\leq \frac{c_{\phi}}{\lambda^{k}},\ \textrm{for all integer}\ k\in [1,5]\ \textrm{and}\ x\in B(\frac{1}{4}),$$
$\nabla_{x}$ is as  below (\ref{equ exponential convergence before pulling back}). Moreover, 
$$|(\Gamma^{-1,\star}\phi_{ijk}-\phi_{ijk}(O))|(x)=|\phi_{ijk}-\phi_{ijk}(O))|(y)\leq \frac{c_{\phi}}{\lambda}\ \textrm{when}\ x\in B(\frac{1}{4})\ (y\in B(\frac{1}{4\lambda})).$$ Therefore
\begin{equation}\label{equ 2}
|\widetilde{\phi}-\widetilde{\phi}(O)|_{C_{x}^{5}[B(\frac{1}{4})]}\leq \frac{c_{\phi}}{\lambda},\ C^{5}_{x}\  \textrm{means  the}\ C^{5}-\textrm{norm in}\  x\textrm{-coordinate}.
\end{equation}
We actually have a $(\widetilde{\phi},\widetilde{\psi})-$monopole on $B(\frac{1}{4})$ of exponential rate arbitrarily close to $1$. To see this, 
for any $1>\theta>0$, choose $p\in (-\frac{5}{2},\theta-\frac{5}{2})$ such that  the condition in Definition 2.21 of \cite{Myself2016} is satisfied. Let $\delta_{0}$ be small enough with respect to $A_{0}$ and $p$,
Theorem 1.13  in \cite{Myself2016} (and the rate of convergence given by the proof of it) is directly applicable. Therefore, let  $\lambda$ be large enough such that $\frac{c_{\phi}}{\lambda}<\frac{\delta_{0}}{2}$ and $\frac{1}{4\lambda}<\frac{R_{0}}{2}$,  there exists a $(\widetilde{\phi},\widetilde{\psi})-$monopole $(A,\sigma)$ over $B(\frac{1}{4})$ i.e.
\begin{equation}\label{equ monopole equation big ball}
F_{A}\wedge\widetilde{\psi}+\star_{\widetilde{g}}d_{A}\sigma=0\ \textrm{over}\ B(\frac{1}{4}),\ \widetilde{g}\ \textrm{is the metric of}\ \widetilde{\phi}.
\end{equation}
Moreover, let  $\delta_{0}$ be even smaller if necessary, by the proof of Theorem 1.13 in section 5 of \cite{Myself2016} (also see Definition 2.9 and Theorem 5.1 therein),  $A$ is of exponential rate $1-\theta$ and order 3 i.e.
\begin{equation}\label{equ exponential convergence before pulling back}
|x|^{l+1}|\nabla_{x}^{l}(A-A_{0})|(x)\leq |x|^{1-\theta},\ \textrm{for any integer}\ l\in [0,3],
\end{equation}
 where $|x|$ is just the usual norm of $x$, and $\nabla_{x}$ is the ordinary derivative (of the components of $A-A_{0}$) under the natural charts as  in Definition 2.10 of \cite{Myself2016} (of course here $\eta\rightarrow \mathbb{S}^{6}$ and $A_{0}$ might be trivialized by more than $2$ coordinate neighbourhoods, but this does not make any difference). 

Pulling back both sides of (\ref{equ monopole equation big ball}) via $\Gamma$, we obtain 
\begin{equation}\label{equ monopole equation small ball}
F_{A^{\star}}\wedge (\lambda^{4}\psi)+\Gamma^{\star}[\star_{\widetilde{g}}(d_{A}\sigma)]=0\ \textrm{over}\ B(\frac{1}{4\lambda}),\ A^{\star}=\Gamma^{\star}A.
\end{equation}
Using  \begin{equation}\label{equ star g is covariant} \Gamma^{\star}[\star_{\widetilde{g}}(d_{A}\sigma)]=\star_{\Gamma^{\star}\widetilde{g}} \Gamma^{\star}(d_{A}\sigma),\ \Gamma^{\star}\widetilde{g}=\lambda^{2}g,\ \textrm{and}\ \star_{\lambda^{2}g}=\lambda^{5}\star_{g}\ (\textrm{see Remark}\ \ref{rmk linear algebra}),
\end{equation}
 where $g\ \textrm{is the metric of}\ \phi$, we obtain 
\begin{equation}
F_{A^{\star}}\wedge \psi+\star_{g}d_{A^{\star}}(\lambda \sigma^{\star})=0,\ \sigma^{\star}=\Gamma^{\star}\sigma.
\end{equation} 
The pair $(A^{\star},\lambda \sigma^{\star})$ is the monopole we desire. Since  $\Gamma^{\star}A_{0}=A_{0}$ (as a connection, see the paragraph above (\ref{equ 1})), the estimate (\ref{equ exponential convergence before pulling back}) means 
\begin{equation}\label{equ exponential convergence after  pulling back}
|y|^{l+1}|\nabla_{y}^{l}(A^{\star}-A_{0})|(y)\leq \lambda^{1-\theta} |y|^{1-\theta},
\end{equation}
where $l$ is as in (\ref{equ exponential convergence before pulling back}), and $\nabla_{y}$ is as under (\ref{equ exponential convergence before pulling back}) but in $y-$coordinate. Since $A_{0}$ is smooth on $\mathbb{S}^{6}(1)$, we directly verify by (\ref{equ exponential convergence after  pulling back}) that 
\begin{equation}\label{equ exponential convergence after  pulling back wrt nablaA0}
|y|^{l+1}|\nabla^{l}_{A_{0}}(A^{\star}-A_{0})|(y)\leq C\lambda^{1-\theta} |y|^{1-\theta},
\end{equation}
where $C$ is a constant depending only on $A_{0}$. 

The proof of Theorem \ref{Thm local solution} is complete.  
\end{proof}

\begin{rmk}\label{rmk linear algebra} Under a fixed coordinate basis, for any $G_{2}-$structure $\phi$, the components of the co-associative form $\psi$ and the associated metric $g$ only depend on the components of $\phi$. Moreover, the dependence is  via a composition only of power functions, fractions, and polynomials in terms of  the components of $\phi$.   Thus we directly verify (\ref{equ monopole equation small ball}) and (\ref{equ star g is covariant}).
\end{rmk}


\small
\begin{thebibliography}{0}
\bibitem{Brezis}H. Brezis,  X. Cabr\'e. \emph{Some simple nonlinear PDE’s without solutions}. Boll. Un. Mat. Ital. 1. 223–262 (1998).
\bibitem{DerekHarland} B. Charbonneau, D. Harland. \emph{Deformations of nearly K\"ahler instantons}. arXiv:1510.07720.  
\bibitem{MySchauder}X.X. Chen, Y.Q. Wang. \emph{$C^{2,\alpha}$-estimate for Monge-Ampere equations with H\"older-continuous right hand side}.  arXiv:1406.5825. To appear  in  Annals of Global Analysis and Geometry. 
\bibitem{DonSegal} S.K. Donaldson, E. Segal.
\emph{Gauge Theory in higher dimensions, II}. from: "Geometry
of special holonomy and related topics", edited by N.C. Leung, S.T.
Yau, Surv. Differ.
Geom. 16, International Press (2011) 1–41.
\bibitem{Foscolo} L. Foscolo. \emph{Deformation theory of periodic monopoles (with singularities)}. arXiv:1411.6946.
 \bibitem{Oliveira} G. Oliveira. \emph{$G_{2}$-monopoles with singularities (examples)}. Unpublished work.
 \bibitem{Oliveira1} G. Oliveira. \emph{Monopoles on the Bryant-Salamon $G_{2}$-manifolds}. Journal of Geometry and Physics, vol. 86 (2014), pp. 599-632, ISSN 0393-0440. 
\bibitem{Oliveira2} G. Oliveira. \emph{Calabi-Yau Monopoles for the Stenzel Metric}. To appear in Communications in Mathematical Physics.
\bibitem{Oliveira3}G. Oliveira. \emph{Monopoles on 3 dimensional AC manifolds}. arXiv:1412.2252.
 \bibitem{SaEarpWalpuski} H. Sa Earp, T. Walpuski. \emph{$G_{2}-$instantons over twisted connected sums}. Geometry  Topology 19 (2015) 1263-1285.
 \bibitem{SW} D.A. Salamon, T. Walpuski. \emph{Notes on the octonians}. arXiv:1005.2820.
 \bibitem{Walpuski1}T. Walpuski. \emph{$G_{2}-$instantons on generalised Kummer constructions}. Geometry and  Topology 17 (2013). 2345-2388.
\bibitem{Myself2016}Y.Q. Wang. \emph{Deformation of singular connections I: $G_{2}-$instantons with point singularities}. arXiv:1602.05701.
\bibitem{XuFeng}F. Xu. \emph{On instantons on nearly K\"ahler 6-manifolds}. Asian. J. Math. 2009. International Press Vol. 13, No. 4, pp. 535-568, December 2009.
\end{thebibliography}
\end{document}